\documentclass[12pt, dvipdfmx]{amsart}
\usepackage[top=35truemm,bottom=35truemm,left=30truemm,right=30truemm]{geometry}

\usepackage{amsthm}
\usepackage{mathdots}
\usepackage{amssymb, amsmath}
\usepackage[dvipdfmx]{graphicx}
\usepackage[legacycolonsymbols]{mathtools} 
\usepackage{fancyhdr}

\newtheorem{dfn}{Definition}[section]
\newtheorem{prop}[dfn]{Proposition}
\newtheorem{lem}[dfn]{Lemma}
\newtheorem{thm}[dfn]{Theorem}

\title{$SL_4$-Kloosterman sum via the Bruhat decomposition}
\author{Suzuho Osonoe,\ Maki Nakasuji}
\date{}

\begin{document}


\thanks{This work was supported by Japan Society for the Promotion of Science, Grant-in-Aid for Scientific Research No. 22K03274(M. Nakasuji).}

\maketitle

\begin{abstract}
  We define the Kloosterman sum for $SL_4$ over the Kloosterman set via the Bruhat decomposition and stratify the Kloosterman set using the reduced word decomposition of the Weyl group element. The Kloosterman sum for an $SL_4$-long word is decomposed into finer parts (called the fine Kloosterman sum), and can be written as a finite sum of a product of two classical Kloosterman sums.
\end{abstract}

%%%%%%%%%%%%%%%%%%%%%%%%%%%%%%%%%%%%%%%%%%%%%%%%%%%%%%%%%%%%%%
\begin{section}{Introduction}

  The classical Kloosterman sum is defined by
  \begin{equation*}
    S(m,n;c)=\sum_{\substack{a,d\pmod c \\ ad\equiv1\pmod c}} e\left(\dfrac{ma+nd}{c}\right),
  \end{equation*}
  where $m,n,c\in \mathbb{Z}$ with $c>0$, and $e(x)=\exp(2\pi ix)$. This exponential sum was first introduced by H.\ D.\ Kloosterman in 1926. Since then, the Kloosterman sum have been deeply studied in such fields as spectral theory, modular forms and Poincaré series. For example, the Kloosterman sum appears as an exponential sum over $\gamma=\bigl(
\begin{smallmatrix}
   a & b \\
   c & d
\end{smallmatrix}
\bigl)\ \in SL_2(\mathbb{Z})$ in the computation of the Fourier coefficients of the classical Poincar\'{e} series.

  An important result concerning this sum is the Weil bound:
  \begin{equation}\label{weil}
    |S(m,n;c)|\leq (m,n,c)^{\frac{1}{2}}c^{\frac{1}{2}}\tau(c),
  \end{equation}
  where $(m,n,c)$ denotes the greatest common divisor, and $\tau(c)$ is the number of positive divisors of $c$. This bound enabled the effective use of Kloosterman sums in the study of various areas such as the Kuznetsov trace formula, Fourier expansions and L-functions.

  The classical Kloosterman sum may be defined as a sum over the Kloosterman set arising from the Bruhat decomposition of the group $SL_2(\mathbb{Z})$.
  In recent years, there have been attempts to generalize the Kloosterman sum to higher rank groups such as $GL_n$ and $SL_n$.
  In these generalizations, it is necessary to replace the original number-theoretic approach with a group-theoretic one, for which Bruhat decompositions and Weyl groups play key roles.
  In the case of $SL_3$, the definition and properties of the generalized Kloosterman sum have been studied, and its analytic structure has been clarified. Kiral and Nakasuji (\cite{KN}) have provided a stratification of the Kloosterman sum for $SL_3$ using the reduced word decomposition of the Weyl group element. Additionally, the Kloosterman sum for an $SL_3$-long word is decomposed into finer parts (and called the fine Kloosterman sum), which can be written as a finite sum of a product of two classical Kloosterman sums \cite[Theorem 5.10]{KN}.
  
  However, the structure of the Kloosterman sum for $SL_4$ has not yet been sufficiently studied.
  In the present paper, we attempt to construct and analyze the Kloosterman sum for an $SL_4$-long word, following the group-theoretic approach used for the $SL_3$ case in \cite{KN}.
  
  The paper is organized as follows. In Section 2, we review the Bruhat decomposition and the Weyl groups, which are essential for the group-theoretic approach to the Kloosterman sum. In particular, we give a preliminary discussion of the reduced word decomposition of the long word that will be used in the later sections. In Section 3, we define the Kloosterman sum for $SL_4$ and state the resulting properties.

\end{section}

%%%%%%%%%%%%%%%%%%%%%%%%%%%%%%%%%%%%%%%%%%%%%%%%%%%%%%%%%%%%%%
\begin{section}{preliminaries}

  The Bruhat decomposition plays a key role in the definition and structure of a Kloosterman sum for $SL_n$. For $G=SL_n$, let $B\subset G$ denote the Borel subgroup consisting of upper triangular matrices. The Bruhat decomposition can be expressed using elements of the Weyl group $W$ as 
  \begin{equation*}
    G=\bigsqcup_{w\in W}BwB.
  \end{equation*}
  In the case of $SL_n$, the Weyl group $W$ can be identified with the symmetric group $S_n$.
  Additionally, any element $w\in W\cong S_n$ can be written as a product of simple transposition $s_i=(i,\ (i+1))$, for $1\leq i \leq n-1.$

  Let
  \begin{equation*}
    w_0=\begin{pmatrix}
   1 & 2 & 3 & \cdots & (n-1) & n \\
   \downarrow & \downarrow & \downarrow & \cdots & \downarrow & \downarrow \\
   n & (n-1) & (n-2) & \cdots & 2 & 1 \\
\end{pmatrix}\in W.
  \end{equation*}
  In the present paper, we stratify the Kloosterman set using the reduced word decomposition of the long word:
  \begin{equation*}
    w_0=s_1(s_2s_1)(s_3s_2s_1)\cdots(s_{n-1}\cdots s_1).
  \end{equation*}
  In terms of the matrix representation of $SL_n$, this is expressed as
  \begin{equation}\label{w_0}
    w_0=\begin{pmatrix}
    & & & & & (-1)^{n-1} \\
    & & & & \iddots &  \\
    & & & -1 & & \\
    & & 1 & & & \\
    & -1 & & & & \\
    1 & & & & & \\
\end{pmatrix}_{n\times n}.
  \end{equation}

  Let $U$ be the group of unipotent matrices and $T$ be the group of diagonal matrices. For a given $A\in UwTU$, we write its decomposition as follows.
  \begin{equation}\label{3.1} 
A=\begin{pmatrix}
  1 & a_{12} & a_{13} & \cdots & a_{1n}\\
   & 1 & a_{23} & \cdots & a_{2n}\\
   & & & \ddots & \vdots\\
   & & & & 1
\end{pmatrix}
w\begin{pmatrix}
  t_1 & & & \\
   & t_2 & & \\
   & & \ddots & \\
   & & & t_n
\end{pmatrix}
\begin{pmatrix}
  1 & b_{12} & b_{13} & \cdots & b_{1n}\\
   & 1 & b_{23} & \cdots & b_{2n}\\
   & & & \ddots & \vdots\\
   & & & & 1
\end{pmatrix}.
  \end{equation}
  The above decomposition is a finer version of the Bruhat decomposition, where $a_{ij},b_{ij}$ and $t_i$ can be expressed in terms of the entries of $A$ (\cite[Proposition 3.1]{KN}). For subsets $I=\{i_1<i_2<\cdots <i_k\}\subseteq\{1,\cdots,n\}$ and $J=\{j_1<j_2<\cdots <j_k\}\subseteq\{1,\cdots,n\}$, let $M_{I,J}$ be the determinant of the submatrix obtained by removing the $I$-th row and the $J$-th column from the original matrix, which we write as $M_{I,J}=\langle \mathbf{e}^*_{I},A\mathbf{e}_J \rangle$. 
  For example, let
  \begin{equation*}
    A=\begin{pmatrix}
      A_{11} & A_{12} & A_{13} & A_{14}\\
      A_{21} & A_{22} & A_{23} & A_{24}\\
      A_{31} & A_{32} & A_{33} & A_{34}\\
      A_{41} & A_{42} & A_{43} & A_{44}
    \end{pmatrix}.
  \end{equation*}
  Then, the minor $M_{234,123}(=\langle \mathbf{e}^*_{2,3,4},A\mathbf{e}_{1,2,3} \rangle)$ means
  \begin{equation*}
    \det\begin{pmatrix}
      A_{21} & A_{22} & A_{23}\\
      A_{31} & A_{32} & A_{33}\\
      A_{41} & A_{42} & A_{43}
    \end{pmatrix}.
  \end{equation*}

  For a given positive root $\alpha$, we define a map $\iota_{\alpha}:SL_2(\mathbb{R})\rightarrow SL_n(\mathbb{R})$ such that its image lies in the rank-one subgroup associated with $\alpha$.

  In the case of $SL_4$, for the simple roots $\alpha=(1,-1,0,0),\ \beta=(0,1,-1,0),\ \gamma=(0,0,1,-1)$ and $x=\left(
  \begin{smallmatrix}
   a&b\\
   c&d
   \end{smallmatrix}
   \right)\in GL_2$, we define
   \begin{equation}
  \iota_\alpha(x)=\begin{pmatrix}
  a&b&&\\
  c&d&&\\
  &&1&\\
  &&&1
\end{pmatrix},\quad \iota_\beta(x)=\begin{pmatrix}
  1&&&\\
  &a&b&\\
  &c&d&\\
  &&&1
\end{pmatrix},\quad \iota_\gamma(x)=\begin{pmatrix}
  1&&&\\
  &1&&\\
  &&a&b\\
  &&c&d
\end{pmatrix}.
\end{equation}
Using $\iota_*$ as above, we use $s_*$ to denote the following:
\begin{equation}\label{2.2}
  s_\alpha=\iota_{\alpha}\left(\begin{pmatrix}
  0 & -1\\
  1 & 0
\end{pmatrix}\right),\quad s_\beta=\iota_{\beta}\left(\begin{pmatrix}
  0 & -1\\
  1 & 0
\end{pmatrix}\right),\quad s_\gamma=\iota_{\gamma}\left(\begin{pmatrix}
  0 & -1\\
  1 & 0
\end{pmatrix}\right) .
\end{equation}
  These $s_*$ terms correspond to the simple transpositions $s_{\alpha}=(12),\ s_{\beta}=(23)$ and $ s_{\gamma}=(34)$ in the symmetric group $S_4$.
  
\end{section}

%%%%%%%%%%%%%%%%%%%%%%%%%%%%%%%%%%%%%%%%%%%%%%%%%%%%%%%%%%%%%%
\begin{section}{$SL_4$-long word Kloosterman sum}

  \subsection{Kloosterman set and Kloosterman sum}

  In this section, following \cite[Section 4]{KN}, we define the Kloosterman sum for $SL_4(\mathbb{Z})$.

  For a vector $\mathbf{c}=\{c_1,c_2,c_3\}\in\mathbb{Z}^3$, we define 
  \begin{equation}\label{t(c)}
    t(\mathbf{c})=\begin{pmatrix}
      c_1 & & & \\
      & \frac{c_2}{c_1} & & \\
      & & \frac{c_3}{c_2} & \\
      & & & \frac{1}{c_3}
    \end{pmatrix},
  \end{equation}
  and
  \begin{equation}\label{omega}
    \Omega_{\omega}(\mathbf{c})=\{A\in SL_4(\mathbb{Z}) \ |\  A=u_L\omega t(\mathbf{c})u_R\in B\omega B,\ u_L,u_R\in U\}.
  \end{equation}

  Let $\mathbf{n}=\{n_1,n_2,n_3\}\in\mathbb{Z}^3$. We use the character $\psi_\mathbf{n}$ to denote
  \begin{equation*}
    \psi_{\mathbf{n}}(u)=e(n_1 u_{1,2}+n_2u_{2,3}+n_3u_{3,4}),
  \end{equation*}
  where $u=(u_{i,j})$ is a unipotent matrix, and $e(x)=\exp(2\pi ix)$.

  For $\mathbf{m},\mathbf{n},\mathbf{c}\in\mathbb{Z}^3$, we define the usual (coarse) Kloosterman sum for $SL_4(\mathbb{Z})$ as
  \begin{equation}
    S_\omega(\mathbf{m},\mathbf{n};\mathbf{c})=\sum_{\substack{A\in U(\mathbb{Z})\backslash\Omega_{\omega}(\mathbf{c})/U(\mathbb{Z})\\A=u_L\omega t(\mathbf{c})u_R}} \psi_{\mathbf{m}}(u_L)\psi_{\mathbf{n}}(u_R).
  \end{equation}

  From here on, we consider the Kloosterman sum for $SL_4$ for the case of $w=w_0$ with \eqref{w_0} which is the long word in the Weyl group, appearing in the decomposition of $A$ in \eqref{omega}.

  For a given $A\in SL_4(\mathbb Z)\cap Bw_0B$, we write
  \begin{equation}\label{A}
    A=\begin{pmatrix}
      A_{11} & A_{12} & A_{13} & A_{14}\\
      A_{21} & A_{22} & A_{23} & A_{24}\\
      A_{31} & A_{32} & A_{33} & A_{34}\\
      A_{41} & A_{42} & A_{43} & A_{44}
    \end{pmatrix}.
  \end{equation}

  In addition, for a given $A\in SL_4 \cap Uw_0TU$, we write $A$ as in \eqref{3.1}:
  \begin{equation}\label{3.1'}
A=\begin{pmatrix}
  1 & a_{12} & a_{13} & a_{14}\\
   & 1 & a_{23} & a_{24}\\
   & & 1 & a_{34} \\
   & & & 1
\end{pmatrix}
w_0\begin{pmatrix}
  t_1 & & & \\
   & t_2 & & \\
   & & t_3 & \\
   & & & t_4
\end{pmatrix}
\begin{pmatrix}
  1 & b_{12} & b_{13} & b_{14}\\
   & 1 & b_{23} & b_{24}\\
   & & 1 & b_{34}\\
   & & & 1
\end{pmatrix}.
  \end{equation}
  The diagonal matrix $T$ in \eqref{3.1'} becomes $t(\mathbf{c})$ in \eqref{t(c)} when
  \begin{equation}\label{c}
    c_1=t_1,\quad c_2=t_1t_2,\quad c_3=t_1t_2t_3.
  \end{equation}

  From \cite[Proposition 3.1]{KN}, we obtain
  \begin{equation}\label{t}
    t_1=A_{41},\quad t_1t_2=M_{34,12}=\begin{vmatrix}
      A_{31} & A_{32}\\
      A_{41} & A_{42}
    \end{vmatrix},\quad t_1t_2t_3=M_{234,123}=\begin{vmatrix}
      A_{21} & A_{22} & A_{23}\\
      A_{31} & A_{32} & A_{33}\\
      A_{41} & A_{42} & A_{43}
    \end{vmatrix}.
  \end{equation}

  \begin{lem}\label{lem5.1}
    For a given $A\in SL_4(\mathbb Z)\cap Bw_0B$ as above, we have
$$\gcd(A_{41},A_{42},A_{43})=\gcd(M_{234,123},M_{134,123},M_{124,123}).$$
  \end{lem}

  \begin{proof}
  Let $f_1=\gcd(A_{41},A_{42},A_{43})$ and $f_2=\gcd(M_{234,123},M_{134,123},M_{124,123})$. Each minor of $f_2$ includes $[A_{41}\ A_{42}\ A_{43}$]. This proves $f_1|f_2$.\par
  For the converse, since $|A|=1$, we can write
    $$A_{41}=\begin{vmatrix}
      M_{234,123} & M_{234,134} & M_{234,124}\\
      M_{134,123} & M_{134,134} & M_{134,124}\\
      M_{124,123} & M_{124,134} & M_{124,124}
    \end{vmatrix},\ A_{42}=\begin{vmatrix}
      M_{234,123} & M_{234,234} & M_{234,124}\\
      M_{134,123} & M_{134,234} & M_{134,124}\\
      M_{124,123} & M_{124,234} & M_{124,124}
    \end{vmatrix},$$
    $$A_{43}=\begin{vmatrix}
      M_{234,123} & M_{234,234} & M_{234,134}\\
      M_{134,123} & M_{134,234} & M_{134,134}\\
      M_{124,123} & M_{124,234} & M_{124,134}
    \end{vmatrix}.$$
    Each minor of $f_2$ appears in the first column of each determinant. This means that $f_2|f_1$.
  \end{proof}

  We set
  \begin{equation}
    f=\gcd(A_{41},A_{42},A_{43})=\gcd(M_{234,123},M_{134,123},M_{124,123}).
  \end{equation}
  We define $D_1,D_2,D_3\in \mathbb{Z}_{>0}$  as follows using \eqref{c} and \eqref{t},
  \begin{equation}\label{5.1}
    c_1=A_{41}=D_1f,\quad c_2=M_{34,12}=D_2f,\quad c_3=M_{234,123}=D_3f.
  \end{equation}
  Furthermore, we define the following set as the fine Kloosterman set
  \begin{eqnarray*}\label{5.2}
    \Omega(D_1,D_2,D_3,f)= \{A\in SL_4(\mathbb{Z})\ |\ \gcd(A_{41},A_{42},A_{43})=f,\ A_{41}=D_1f,\quad\quad\quad\quad\quad\\ M_{34,12}=D_2f,\ M_{234,123}=D_3f\}.
  \end{eqnarray*}

  \begin{dfn}
    For $\mathbf{m},\mathbf{n}\in\mathbb{Z}^3$, we define the fine Kloosterman sum for $SL_4$ as follows, using $D_1,D_2,D_3$ and $f$ defined above.
    \begin{equation}\label{5.4}
      \mathcal{S}_{\omega_0}(\mathbf{m},\mathbf{n};D_1,D_2,D_3,f)=\sum_{\substack{A\in U(\mathbb{Z})\backslash\Omega(D_1,D_2,D_3,f)/U(\mathbb{Z})\\A=u_L\omega t(D_1f,D_2f,D_3f)u_R}} \psi_{\mathbf{m}}(u_L)\psi_{\mathbf{n}}(u_R).
    \end{equation}
  \end{dfn}

  Similarly as before, the (usual) coarse Kloosterman sum for $SL_4$ can thus be written as the sum
  \begin{equation}\label{fine}
    S_{\omega_0}(\mathbf{m},\mathbf{n};(c_1,c_2,c_3))=\sum_{f|(c_1,c_2,c_3)} \mathcal{S}_{\omega_0}\left(\mathbf{m},\mathbf{n};\dfrac{c_1}{f},\dfrac{c_2}{f},\dfrac{c_3}{f},f\right).
  \end{equation}
  From \eqref{fine}, the properties of the fine Kloosterman sum are expected to give the properties of (usual) coarse Kloosterman sum.

%-------------------------------------------------------------
  \subsection{Properties of fine Kloosterman sum}

  In this subsection, we consider the parametrization of the fine Kloosterman cells, and discuss the properties of the fine Kloosterman sum. Using the reduced word decomposition $\omega_0=s_\alpha s_\beta s_\alpha s_\gamma s_\beta s_\alpha$ with the notation in \eqref{2.2}, we parametrize the fine Kloosterman sets $\Omega(D_1,D_2,D_3,f)$:
  \begin{equation}\label{5.6}
    A=\iota_{\alpha}(\gamma_1)\iota_{\beta}(\gamma_3)\iota_{\alpha}(\gamma_2)\iota_{\gamma}(\gamma_6)\iota_{\beta}(\gamma_5)\iota_{\alpha}(\gamma_4),
  \end{equation}
  where
  \begin{equation*}
    \gamma_i=\begin{pmatrix}
      x_i & b_i\\
      d_i & y_i
    \end{pmatrix}\ (\text{if}\ 1\leq i \leq 5),\quad
    \gamma_6=\begin{pmatrix}
      x_6 & b_6\\
      f & y_6
    \end{pmatrix}.
  \end{equation*}
  Then we can write
  \begin{equation}\label{D}
    D_1=d_4d_5,\quad D_2=d_2d_3d_5,\quad D_3=d_1d_3.
  \end{equation}

  The calculation of \eqref{5.6} leads to
  $$A_{41}=d_4d_5f,\quad A_{42}=d_5fy_4,\quad A_{43}=fy_5,$$
  $$M_{234,123}=d_1d_3f,\quad M_{134,123}=d_3fx_1,\quad M_{124,123}=fx_3.$$

  If
  $$\gcd(A_{41},A_{42})=d_5f,\quad \gcd(M_{234,123},M_{134,123})=d_3f,$$
  then $\gcd(y_4,d_4)=1$,\ $\gcd(x_1,d_1)=1$,\ $\gcd(y_5,d_5)=1$ and $\gcd(x_3,d_3)=1$. In other words, $y_4$ is defined up to modulo $d_4$, $x_1$ is defined up to modulo $d_1$, $y_5$ is defined up to modulo $d_5$ and $x_3$ is defined up to modulo $d_3$.

  \begin{lem}\label{lem5.3}
    For $A\in SL_4(\mathbb{Z})$, we have $A_{44}M_{123,123}\equiv 1 \pmod f$. In particular, $\gcd(A_{44},f)=\gcd(M_{123,123},f)=1.$
  \end{lem}

  \begin{proof}
    Expanding the determinant of $A$ along the last row, we obtain
    \begin{eqnarray*}
      |A|&=&-A_{41}M_{123,234}+A_{42}M_{123,134}-A_{43}M_{123,134}+A_{44}M_{123,123}\\
      &=& f\{-{d_1}'M_{123,234}+(A_{42}/f)M_{123,134}-(A_{43}/f)M_{123,124}\}+A_{44}M_{123,123}.
    \end{eqnarray*}
  \end{proof}

  The calculation of \eqref{5.6} leads to $A_{44}=y_6$ and $M_{123,123}=x_6$. Therefore, by Lemma \ref{lem5.3}, we have $x_6y_6\equiv 1\pmod f$.

  From \cite[Proposition 3.1]{KN}, for $A\in SL_4(\mathbb{Z})\cap u_L\omega_0Tu_R$, we know that
  $$u_L=\begin{pmatrix}
    1 & \dfrac{\langle e^*_{1,3,4},Ae_{1,2,3}\rangle}{t_1t_2t_3} & \dfrac{\langle e^*_{1,4},Ae_{1,2}\rangle}{t_1t_2} & \dfrac{\langle e^*_{1},Ae_{1}\rangle}{t_1}\\
    & 1 & \dfrac{\langle e^*_{2,4},Ae_{1,2}\rangle}{t_1t_2} & \dfrac{\langle e^*_{2},Ae_{1}\rangle}{t_1}\\
    & & 1 & \dfrac{\langle e^*_{3},Ae_{1}\rangle}{t_1}\\
    & & & 1
  \end{pmatrix},$$
  and
  $$u_R=\begin{pmatrix}
    1 & \dfrac{\langle e^*_{4},Ae_{2}\rangle}{t_1} & \dfrac{\langle e^*_{4},Ae_{3}\rangle}{t_1} & \dfrac{\langle e^*_{4},Ae_{4}\rangle}{t_1}\\
    & 1 & \dfrac{\langle e^*_{3,4},Ae_{1,3}\rangle}{t_1t_2} & \dfrac{\langle e^*_{3,4},Ae_{1,4}\rangle}{t_1t_2} \\
    & & 1 & \dfrac{\langle e^*_{2,3,4},Ae_{1,2,4}\rangle}{t_1t_2t_3} \\
    & & & 1
  \end{pmatrix}.$$

  Defining $u_1=d_1x_2+d_2x_3y_1,\ u_2=d_2x_4+d_4x_5y_2,\ v_1=d_1x_2y_3+d_2y_1,\ v_2=d_2x_4y_5+d_4y_2$, we have the following proposition.

  \begin{prop}\label{prop5.6}
    For a given $A\in SL_4(\mathbb{Z})$, we choose $d_1,d_2,d_3,d_4,d_5,f$ as in \eqref{5.1} and \eqref{D}. Then for $u_1,u_2,v_1,v_2\in\mathbb{Z}$ as above, the Bruhat decomposition of $A$ can be written as
    \begin{equation}\label{5.9}
      A=u_L\begin{pmatrix}
        &&&-1\\
        &&1&\\
        &-1&&\\
        1&&&
      \end{pmatrix}\begin{pmatrix}
        d_4d_5f&&&\\
        &\frac{d_2d_3}{d_4}&&\\
        &&\frac{d_1}{d_2d_5}&\\
        &&&\frac{1}{d_1d_3f}
      \end{pmatrix}u_R,
    \end{equation}
    where
    $$u_L=\begin{pmatrix}
      1 & \dfrac{x_1}{d_1} & \dfrac{x_1u_1-d_2x_3}{d_1d_2d_3} & \dfrac{x_1S-d_2x_3T+d_2d_4d_5x_6}{d_1d_2d_3d_4d_5f}\\
      & 1 & \dfrac{u_1}{d_2d_3} & \dfrac{S}{d_2d_3d_4d_5f}\\
      & & 1 & \dfrac{T}{d_4d_5f}\\
      & & & 1
    \end{pmatrix}$$
    and
    $$u_R=\begin{pmatrix}
      1 & \dfrac{y_4}{d_4} & \dfrac{y_5}{d_4d_5} & \dfrac{y_6}{d_4d_5f}\\
      & 1 & \dfrac{v_2}{d_2d_5} & \dfrac{d_4d_5y_3+d_3y_6u_2}{d_2d_3d_5f}\\
      & & 1 & \dfrac{d_5v_1+d_1d_3x_5y_6}{d_1d_3f}\\
      & & & 1
    \end{pmatrix},$$
    in which $T=d_3u_2+d_4d_5x_6y_3,\ S=d_2d_4d_5x_6y_1(x_3y_3-1)+d_3(u_1u_2-d_1d_4x_5)$.
  \end{prop}

  \begin{prop}
    For given nonzero integers $d_1,d_2,d_3,d_4,d_5,f$ and $x_1\pmod{d_1}$,\ $y_4\pmod{d_4}$,\ $x_3\pmod{d_3}$,\ $y_5\pmod{d_5}$, the product in \eqref{5.9} gives rise to an integral matrix if and only if the following congruence conditions are satisfied.
  \begin{equation}\label{A11}
    d_2d_4d_5x_6-Td_2x_3+Sx_1\equiv0 \pmod{d_1d_2d_3}
  \end{equation}
  \begin{equation}\label{A21}
    S\equiv0 \pmod{d_2d_3}
  \end{equation}
  \begin{equation}\label{A12}
    y_4(d_2d_4d_5x_6-Td_2x_3+Sx_1)+d_2d_3(d_2x_3-x_1u_1)\equiv0 \pmod{d_1d_2d_3d_4}
  \end{equation}
  \begin{equation}
    Sy_4-d_2d_3u_1\equiv0 \pmod{d_2d_3d_4}
  \end{equation}
  \begin{equation}\label{A32}
    Ty_4-d_2d_3\equiv0 \pmod{d_4}
  \end{equation}
  \begin{equation}
    y_5(d_2d_4d_5x_6-Td_2x_3+Sx_1)+v_2d_3(d_2x_3-x_1u_1)+d_1d_3d_4x_1\equiv0 \pmod{d_1d_2d_3d_4d_5}
  \end{equation}
  \begin{equation}
    d_1d_3d_4-d_3u_1v_2+Sy_5\equiv0 \pmod{d_2d_3d_4d_5}
  \end{equation}
  \begin{equation}\label{A33}
    Ty_5-d_3v_2\equiv0 \pmod{d_4d_5}
  \end{equation}
  \begin{equation}
    -(d_3u_2y_6+d_4d_5y_3)+Ty_6\equiv0 \pmod{d_4d_5f}
  \end{equation}
  \begin{equation}\label{A24}
    d_1d_3d_4x_5y_6-u_1(d_3u_2y_6+d_4d_5y_3)+d_4d_5v_1+Sy_6\equiv0 \pmod{d_2d_3d_4d_5f}
  \end{equation}
  \begin{eqnarray}\label{A14}
    x_1(d_1d_3d_4x_5y_6-u_1(d_3u_2y_6+d_4d_5y_3)+d_4d_5v_1+Sy_6)+d_2d_4d_5x_3y_3+d_2d_4d_5\quad\quad\nonumber\\\times(x_6y_6-1)+d_2d_3u_2x_3y_6-Td_2x_3y_6\equiv0\pmod{d_1d_2d_3d_4d_5f}
  \end{eqnarray}
  \end{prop}

  In what follows, we will assume $x_1$ and $y_4$ are chosen to be relatively prime to $d_1d_2d_3d_4d_5f$ and we will denote the inverses of $x_1$ and $y_4$ by $\bar{x_1}$ and $\bar{y_4}$, respectively.

  \begin{lem}\label{prop5.7}
    For $x_1,x_3,y_4$ relatively prime to $d_1d_2d_3d_4d_5f$ such that $x_1\bar{x_1}\equiv1,\ y_4\bar{y_4}\equiv1$ and $x_3y_3\equiv1$, the following congruence conditions are satisfied.
    \begin{eqnarray}
      d_2d_3u_1 &\equiv& d_2^2d_3\bar{x_1}x_3 \pmod{d_1d_2d_3} \label{1}\\
      d_3u_2+d_4d_5x_6y_3 &\equiv& d_2d_3\bar{y_4} \pmod{d_4} \label{2}\\
      d_3v_2 &\equiv& d_2d_3\bar{y_4}y_5 \pmod{d_4} \label{3}\\
      d_2d_3d_4(d_5v_1+d_1d_3x_5y_6) &\equiv& d_2^2d_3d_4d_5\bar{x_1} \pmod{d_2d_3} \label{4}
    \end{eqnarray}
  \end{lem}

  \begin{proof}
    Substituting \eqref{A11} into \eqref{A12} and multiplying both sides by $\bar{x_1}$, we obtain \eqref{1}. For \eqref{2}, it follows from multiplying both sides of \eqref{A32} by $\bar{y_4}$.
    
    Substituting the obtained conditions on $d_3u_2$ into \eqref{A33} yields \eqref{3} as follows.
    \begin{eqnarray*}
      d_3u_2y_5+d_4d_5x_6y_3-d_3v_2 &\equiv&0 \pmod{d_4d_5}\\
      d_3v_2 &\equiv& d_2d_3\bar{y_4}y_5 \pmod{d_4}.
    \end{eqnarray*}

    For \eqref{4}, we use the assumption $x_3y_3\equiv1 \pmod{d_1d_2d_3d_4d_5f}$. Substituting \eqref{A21} into \eqref{A24} and multiplying both side by $d_2$, the term $d_2d_3u_2$ is expressed as
    \begin{equation}\label{5}
      d_2d_4(d_5v_1+d_1d_3x_5y_6)\equiv d_2d_3u_2 \cdot  y_6u_1+d_2d_4d_5y_3u_1 \pmod{d_2d_3}.
    \end{equation}
    Next, we consider the conditions on $d_2d_3u_2$. Substituting \eqref{A21} into \eqref{A11} and applying the assumption yields the following:
    \begin{eqnarray*}
      d_2d_4d_5x_6-d_2d_3x_3u_2-d_2d_4d_5x_3x_6y_3 &\equiv& 0 \pmod{d_2d_3}\\
      d_2d_3u_2 &\equiv& 0 \pmod{d_2d_3}.
    \end{eqnarray*}
    Substituting the above condition into \eqref{5}, and then applying $d_2d_3u_1$, we obtain \eqref{4} as
    \begin{eqnarray*}
       d_2d_3d_4(d_5v_1+d_1d_3x_5y_6) &\equiv& d_2d_3u_1\cdot d_4d_5y_3 \pmod{d_2d_3}\\
       d_2d_3d_4(d_5v_1+d_1d_3x_5y_6) &\equiv& d_2^2d_3d_4d_5\bar{x_1} \pmod{d_2d_3}.
    \end{eqnarray*}
  \end{proof}

  \begin{thm}\label{thm5.10}
    Let $\mathbf{m}=\{m_1,m_2,m_3\},\ \mathbf{n}=\{n_1,n_2,n_3\} \in \mathbb{Z}^3$. The fine Kloosterman sum $\mathcal{S}_{\omega_0}(\mathbf{m},\mathbf{n};D_1,D_2,D_3,f)=\mathcal{S}_{\omega_0}(\mathbf{m},\mathbf{n};d_4d_5,d_2d_3d_5,d_1d_3,f)$ with the notation in \eqref{D} becomes
    \begin{multline*}
    d_1^3d_2^2d_3^2d_4^2d_5^4f^4\\
    \times\sum_{\substack{x_3\pmod{d_3}\\y_5\pmod{d_5}}}S(m_1,(m_2d_1fx_3+n_3d_2d_5)/d_3f;d_1) \cdot S(n_1,(n_2d_4fy_5+m_3d_2d_3)/d_5f;d_4),
    \end{multline*}
    when it satisfies $m_2d_1\equiv0\pmod{d_2d_3},\ m_3\equiv0\pmod{d_5f}\ n_2d_4\equiv0\pmod{d_2d_3d_5}$ and $n_3\equiv0\pmod{d_1d_3d_4f}$.
  \end{thm}

  \begin{proof}
  We calculate using the definition of the fine Kloosterman sum along with the notation in \eqref{D} as follows:
  \begin{eqnarray*}
    \mathcal{S}_{\omega_0}(\mathbf{m},\mathbf{n};D_1,D_2,D_3,f) &=& \mathcal{S}_{\omega_0}(\mathbf{m},\mathbf{n};d_4d_5,d_2d_3d_5,d_1d_3,f)\\
    &=& \sum_{\substack{A\in U(\mathbb{Z})\backslash\Omega(d_4d_5,d_2d_3d_4,d_1d_3,f)/U(\mathbb{Z})\\A=u_L\omega t(d_4d_5f,d_2d_3d_5f,d_1d_3f)u_R}} \psi_{\mathbf{m}}(u_L)\psi_{\mathbf{n}}(u_R).
  \end{eqnarray*}
  Let us express superdiagonal elements in the unipotent factors of the Bruhat decomposition in terms of these coordinates as in Proposition \ref{prop5.6}. Then $\psi_{\mathbf{m}}(u_L)\psi_{\mathbf{n}}(u_R)$ becomes\\
    \begin{equation*}
    e\left(\frac{m_1x_1}{d_1}+\frac{m_2u_1}{d_2d_3}+\frac{m_3(d_3u_2+d_4d_5x_6y_3)}{d_4d_5f}+\frac{n_1y_4}{d_4}+\frac{n_2v_2}{d_2d_5}+\frac{n_3(d_5v_1+d_1d_3x_5y_6)}{d_1d_3f}\right).
  \end{equation*}
  Then using Lemma \ref{prop5.7}, we substitute the values of $u$ and $v$ into the above to compute\\
  \begin{multline*}
  \mathcal{S}_{\omega_0}(\mathbf{m},\mathbf{n};d_4d_5,d_2d_3d_5,d_1d_3,f)\\
  =\sum_{\substack{x_3\pmod{d_3}\\y_5\pmod{d_5}}}
  \sum_{\substack{x_1\pmod{d_1}\\x_1\bar{x_1}\equiv1\pmod{d_1}}}e\left(\frac{m_1x_1+\{(m_2d_1fx_3+n_3d_2d_5)/d_3f\}\bar{x_1}}{d_1}\right)\\
  \times\sum_{\substack{y_4\pmod{d_4}\\y_4\bar{y_4}\equiv1\pmod{d_4}}}e\left(\frac{n_1y_4+\{(n_2d_4fy_5+m_3d_2d_3)/d_5f\}\bar{y_4}}{d_4}\right)\\
  \times\sum_{k_1=0}^{d_4d_5f-1}e\left(\frac{m_2d_1}{d_2d_3}k_1\right)
  \sum_{k_2=0}^{d_1d_2d_3d_5f-1}e\left(\frac{m_3}{d_5f}k_2\right)\\
  \times\sum_{k_3=0}^{d_1d_2d_3d_5f-1}e\left(\frac{n_2d_4}{d_2d_3d_5}k_3\right)
  \sum_{k_4=0}^{d_1d_4d_5f-1}e\left(\frac{n_3}{d_1d_3d_4f}k_4\right).
  \end{multline*}

    Here, the congruence conditions on Lemma \ref{prop5.7} is arranged so that the modulus is $d_1d_2d_3d_4d_5f$. 
    Summation over $k_i$ gives the congruence condition is
    $m_2d_1\equiv0\pmod{d_2d_3},\ m_3\equiv0\pmod{d_5f},\ n_2d_4\equiv0\pmod{d_2d_3d_5},\ n_3\equiv0\pmod{d_1d_3d_4f},$
    and the product of two classical Kloosterman sums as follows:
    \begin{multline*}
    d_1^3d_2^2d_3^2d_4^2d_5^4f^4\\
    \times\sum_{\substack{x_3\pmod{d_3}\\y_5\pmod{d_5}}}S(m_1,(m_2d_1fx_3+n_3d_2d_5)/d_3f;d_1) \cdot S(n_1,(n_2d_4fy_5+m_3d_2d_3)/d_5f;d_4).
    \end{multline*}
  \end{proof}

  From \eqref{fine}, using Theorem \ref{thm5.10} and the Weil bound in \eqref{weil}, we can obtain a bound on the $SL_4$-long word Kloosterman sum.

  \begin{prop}
    For given $\mathbf{m},\mathbf{n}\in\mathbb{Z}^3$, and $c_1,c_2,c_3>0$, we may bound the long-word Kloosterman sum as
    $$|S_{\omega_0}(\mathbf{m},\mathbf{n};(c_1,c_2,c_3))|\leq c_1^3c_2^2c_3^2(c_2+1)(m_1,c_3)^{\frac{1}{2}}(n_1,c_1)^{\frac{1}{2}}\sqrt{c_1c_3}\tau((c_1,c_2,c_3))\tau(c_1)\tau(c_3),$$
    where $\tau(c)$ is the divisor function.
  \end{prop}

  \begin{proof}
    For the given decomposition of a fine Kloosterman sum as a sum of classical Kloosterman sums in Theorem \ref{thm5.10}, using the Weil bound in \eqref{weil} on the individual sums with attention to \eqref{fine}, we obtain
    \begin{multline*}
      |S_{\omega_0}(\mathbf{m},\mathbf{n};(c_1,c_2,c_3))| \leq \sum_{f|(c_1,c_2,c_3)}d_1^3d_2^2d_3^2d_4^2d_5^4f^4\\ \times\sum_{\substack{x_3\pmod {d_3}\\y_5\pmod{d_5}}} (m_1,d_1)^{\frac{1}{2}} (n_1,d_4)^{\frac{1}{2}} \sqrt{d_1d_4} \tau(d_1)\tau(d_4)
    \end{multline*}
    Each $d$ can be bounded above by $c$ with \eqref{5.1} and \eqref{D}, and the number of divisors $f$ is bounded by the divisor function.
  \end{proof}

  %-------------------------------------------------------------
  \subsection{Parametrization of other group}

  We now consider the case of the symplectic group $Sp(4,\mathbb{Z})$. The symplectic group is the group consisting of all $2n\times 2n$ symplectic matrices that preserve the standard symplectic form, with matrix multiplication as the group operation, and is defined as follows:
  \begin{equation}
    Sp(2n,\mathbb{Z})=\{ A\in GL_{2n}(\mathbb{Z})\ |\ A^\top JA=J\},
  \end{equation}
  where $A^\top$ is the transpose of $A$. $J$ is defined as
  $$J=\begin{pmatrix}
    0 & I_n\\
    -I_n & 0
  \end{pmatrix},$$
  and $I_n$ is the identity matrix.

  For example, $Sp(2,\mathbb{Z})$ is the group of $2\times 2$ symplectic matrices, which coincides with the special linear group $SL_2(\mathbb{Z})$. That is,
  $$Sp(2,\mathbb{Z})=SL_2(\mathbb{Z})=\{A\in GL_2(\mathbb{Z})\ |\ \det(A)=1\}.$$
  For $n\geq2$, it is convenient to write matrices in block form. For $n=2$ this block form is given by
  \begin{equation}
    A=\begin{pmatrix}\label{Sp4}
    P & Q\\
    R & S
  \end{pmatrix},
  \end{equation}
  where $P,Q,R,S\in M_2(\mathbb{Z})$, and the symplectic condition $A^\top JA=J$ with J=$\bigl(\begin{smallmatrix}
    0 & I_n\\
    -I_n & 0
  \end{smallmatrix}\bigl)$ is equivalent to the following conditions:
  \begin{equation*}
    P^\top R=R^\top P,\quad Q^\top S=S^\top Q,\quad P^\top S-R^\top Q=I_2.
  \end{equation*}
  By applying \eqref{Sp4} to \eqref{A} we obtain the following conditions on the minors of $A$. In the case of $Sp(4,\mathbb{Z})$, we have following conditions:
  \begin{eqnarray*}
    M_{13,12}+M_{24,12}=0,\quad M_{13,34}+M_{24,34}=0,\\
    M_{13,13}+M_{24,13}=1,\quad M_{13,14}+M_{24,14}=0,\\
    M_{13,23}+M_{24,23}=0,\quad M_{13,24}+M_{24,24}=1.
  \end{eqnarray*}

  We now consider the orthogonal group $O(n)$. It has a subgroup consisting of all matrices with determinant $1$, called the special orthogonal group, denoted by $SO(n)$. It is defined as
  \begin{equation*}
    SO(n)=\{A\in GL_n(\mathbb{Z})\ |\ A^\top A=I_n,\ \det(A)=1\}.
  \end{equation*}
  From this definition, we can derive the conditions on the entries of a matrix $A$. In the case of $SO(4,\mathbb{Z})$, we have following conditions:
  \begin{eqnarray*}
  A_{1i}^2+A_{2i}^2+A_{3i}^2+A_{4i}^2&=&1,\quad (1\leq i\leq4),\\
  A_{11}A_{1i}+A_{21}A_{2i}+A_{31}A_{3i}+A_{41}A_{4i}&=&0,\quad (2\leq i\leq4),\\
  A_{12}A_{1i}+A_{22}A_{2i}+A_{32}A_{3i}+A_{42}A_{4i}&=&0,\quad (3\leq i\leq4),\\
  A_{13}A_{14}+A_{23}A_{24}+A_{33}A_{34}+A_{43}A_{44}&=&0.
\end{eqnarray*}

  Although both $Sp(4,\mathbb{Z})$ and $SO(4)$ have determinant $1$, but additional conditions are imposed on the minors and entries of $A$. Therefore, the definition of the Kloosterman set and its parametrization as in \eqref{5.6} are expected to be more complicated than in the case of $SL_4(\mathbb{Z})$.
  
\end{section}

%%%%%%%%%%%%%%%%%%%%%%%%%%%%%%%%%%%%%%%%%%%%%%%%%%%%%%%%%%%%%%
\begin{section}{$SL_5$-long word Kloosterman sum}

In this section, we outline a partial study of the fine Kloosterman sum for $SL_5(\mathbb{Z})$ following the cases of $SL_3(\mathbb{Z})$ and $SL_4(\mathbb{Z})$.

\subsection{Kloosterman set and Kloosterman sum}

For a vector $\mathbf{c}=\{c_1,c_2,c_3,c_4\}\in\mathbb{Z}^4$, we define
  \begin{equation}\label{t(c)5}
    t(\mathbf{c})=\begin{pmatrix}
      c_1 & & & &\\
      & \frac{c_2}{c_1} & & &\\
      & & \frac{c_3}{c_2} & &\\
      & & & \frac{c_4}{c_3}&\\
      & & & & \frac{1}{c_4}
    \end{pmatrix},
  \end{equation}
  and
  \begin{equation*}\label{omega5}
    \Omega_{\omega}(\mathbf{c})=\{A\in SL_5(\mathbb{Z}) \ |\  A=u_L\omega t(\mathbf{c})u_R\in B\omega B,\ u_L,u_R\in U\}.
  \end{equation*}

Let $\mathbf{n}=\{n_1,n_2,n_3,n_4\}\in\mathbb{Z}^4$. We use the character $\psi_\mathbf{n}$ to denote
  \begin{equation*}
    \psi_{\mathbf{n}}(u)=e(n_1 u_{1,2}+n_2u_{2,3}+n_3u_{3,4}+n_3u_{4,5}).
  \end{equation*}

  For $\mathbf{m},\mathbf{n},\mathbf{c}\in\mathbb{Z}^4$, we define the usual (coarse) Kloosterman sum for $SL_5(\mathbb{Z})$ as
  \begin{equation*}
    S_\omega(\mathbf{m},\mathbf{n};\mathbf{c})=\sum_{\substack{A\in U(\mathbb{Z})\backslash\Omega_{\omega}(\mathbf{c})/U(\mathbb{Z})\\A=u_L\omega t(\mathbf{c})u_R}} \psi_{\mathbf{m}}(u_L)\psi_{\mathbf{n}}(u_R).
  \end{equation*}

For the case of $\omega=\omega_0$ with \eqref{w_0}, we write $A\in SL_5(\mathbb Z)\cap Bw_0B$ as
  \begin{equation*}\label{A5}
    A=\begin{pmatrix}
      A_{11} & A_{12} & A_{13} & A_{14} & A_{15}\\
      A_{21} & A_{22} & A_{23} & A_{24} & A_{25}\\
      A_{31} & A_{32} & A_{33} & A_{34} & A_{35}\\
      A_{41} & A_{42} & A_{43} & A_{44} & A_{45}\\
      A_{51} & A_{52} & A_{53} & A_{54} & A_{55}
    \end{pmatrix}.
  \end{equation*}
In addition, considering the decomposition $A\in SL_5(\mathbb{Z}) \cap Uw_0TU$ in \eqref{3.1}, the diagonal matrix $T$ becomes $t(\mathbf{c})$ in \eqref{t(c)5} when
  \begin{equation}\label{c5}
    c_1=t_1,\quad c_2=t_1t_2,\quad c_3=t_1t_2t_3,\quad c_4=t_1t_2t_3t_4.
  \end{equation}

  From \cite[Proposition 3.1]{KN}, we obtain
  \begin{equation}\label{t5}
    t_1=A_{51},\quad t_1t_2=M_{45,12},\quad t_1t_2t_3=M_{345,123},\quad t_1t_2t_3t_4=M_{2345,1234}.
  \end{equation}

  \begin{lem}\label{lem5.1}
    For a given $A\in SL_5(\mathbb Z)\cap Bw_0B$ as above, we have
$$\gcd(A_{51},A_{52},A_{53},A_{54})=\gcd(M_{2345,1234},M_{1235,1234},M_{1245,1234},M_{1345,1234}).$$
  \end{lem}

  We set
  \begin{equation*}
    f=\gcd(A_{51},A_{52},A_{53},A_{54})=\gcd(M_{2345,1234},M_{1235,1234},M_{1245,1234},M_{1345,1234}).
  \end{equation*}
  We define $D_1,D_2,D_3,D_4\in \mathbb{Z}_{>0}$  as follows using \eqref{c5} and \eqref{t5},
  \begin{equation}\label{5.15}
    c_1=A_{51}=D_1f,\ c_2=M_{45,12}=D_2f,\ c_3=M_{345,123}=D_3f,\ c_3=M_{2345,1234}=D_4f.
  \end{equation}
  Furthermore, we define the following set as the fine Kloosterman set
  \begin{eqnarray*}\label{5.2}
    \Omega(D_1,D_2,D_3,D_4,f)= \{A\in SL_5(\mathbb{Z})\ |\ \gcd(A_{51},A_{52},A_{53},A_{54})=f,\ A_{51}=D_1f,\quad\quad\quad\\ M_{45,12}=D_2f,\ M_{345,123}=D_3f,\ M_{2345,1234}=D_4f\}.
  \end{eqnarray*}

  \begin{dfn}
    For $\mathbf{m},\mathbf{n}\in\mathbb{Z}^4$, we define the fine Kloosterman sum for $SL_5$ as follows, using $D_1,D_2,D_3,D_4$ and $f$ defined above.
    \begin{equation}\label{5.45}
      \mathcal{S}_{\omega_0}(\mathbf{m},\mathbf{n};D_1,D_2,D_3,D_4,f)=\sum_{\substack{A\in U(\mathbb{Z})\backslash\Omega(D_1,D_2,D_3,D_4,f)/U(\mathbb{Z})\\A=u_L\omega t(D_1f,D_2f,D_3f,D_4f)u_R}} \psi_{\mathbf{m}}(u_L)\psi_{\mathbf{n}}(u_R).
    \end{equation}
  \end{dfn}

  Similarly as before, the (usual) coarse Kloosterman sum for $SL_5$ can thus be written as the sum
  \begin{equation*}\label{fine5}
    S_{\omega_0}(\mathbf{m},\mathbf{n};(c_1,c_2,c_3,c_4))=\sum_{f|(c_1,c_2,c_3,c_4)} \mathcal{S}_{\omega_0}\left(\mathbf{m},\mathbf{n};\dfrac{c_1}{f},\dfrac{c_2}{f},\dfrac{c_3}{f},\dfrac{c_4}{f},f\right).
  \end{equation*}

\subsection{Parametrization of fine Kloosterman cells}

In this subsection, we consider the parametrization of the fine Kloosterman cells. In the case of $SL_5$, for the simple roots $\alpha=(1,-1,0,0,0),\ \beta=(0,1,-1,0,0),\ \gamma=(0,0,1,-1,0),\ \delta=(0,0,0,1,-1)$ and $x=\left(
  \begin{smallmatrix}
   a&b\\
   c&d
   \end{smallmatrix}
   \right)\in GL_2$, we define
   \begin{equation*}
  \iota_\alpha(x)=\begin{pmatrix}
  a&b&&&\\
  c&d&&&\\
  &&1&&\\
  &&&1&\\
  &&&&1
\end{pmatrix},\quad \iota_\beta(x)=\begin{pmatrix}
  1&&&&\\
  &a&b&&\\
  &c&d&&\\
  &&&1&\\
  &&&&1
\end{pmatrix},
\end{equation*}
\begin{equation*}
  \iota_\gamma(x)=\begin{pmatrix}
  1&&&&\\
  &1&&&\\
  &&a&b&\\
  &&c&d&\\
  &&&&1
\end{pmatrix},\quad \iota_\delta(x)=\begin{pmatrix}
  1&&&&\\
  &1&&&\\
  &&1&&\\
  &&&a&b\\
  &&&c&d
\end{pmatrix}.
\end{equation*}

Using $\iota_*$ as above, we use $s_*$ to denote the following:
\begin{equation}\label{2.25}
  s_\alpha=\iota_{\alpha}\left(\begin{pmatrix}
  0 & -1\\
  1 & 0
\end{pmatrix}\right),\ s_\beta=\iota_{\beta}\left(\begin{pmatrix}
  0 & -1\\
  1 & 0
\end{pmatrix}\right),\ s_\gamma=\iota_{\gamma}\left(\begin{pmatrix}
  0 & -1\\
  1 & 0
\end{pmatrix}\right),\ s_\delta=\iota_{\delta}\left(\begin{pmatrix}
  0 & -1\\
  1 & 0
\end{pmatrix}\right).
\end{equation}

Using the reduced word decomposition $\omega_0=s_\alpha s_\beta s_\alpha s_\gamma s_\beta s_\alpha s_\delta s_\gamma s_\beta s_\alpha$ with the notation in \eqref{2.25}, we parametrize the fine Kloosterman sets $\Omega(D_1,D_2,D_3,D_4,f)$:
  \begin{equation}\label{5.65}
    A=\iota_{\alpha}(\gamma_1)\iota_{\beta}(\gamma_3)\iota_{\alpha}(\gamma_2)\iota_{\gamma}(\gamma_6)\iota_{\beta}(\gamma_5)\iota_{\alpha}(\gamma_4)\iota_{\delta}(\gamma_{10})\iota_{\gamma}(\gamma_9)\iota_{\beta}(\gamma_8)\iota_{\alpha}(\gamma_7),
  \end{equation}
  where
  \begin{equation*}
    \gamma_i=\begin{pmatrix}
      x_i & b_i\\
      d_i & y_i
    \end{pmatrix}\ (\text{if}\ 1\leq i \leq 9),\quad
    \gamma_{10}=\begin{pmatrix}
      x_{10} & b_{10}\\
      f & y_{10}
    \end{pmatrix}.
  \end{equation*}
  Then we can write
  \begin{equation}\label{D5}
    D_1=d_7d_8d_9,\quad D_2=d_4d_5d_6d_8d_9,\quad D_3=d_2d_3d_5d_6d_9,\quad D_4=d_1d_3d_6.
  \end{equation}

  From \cite[Proposition 3.1]{KN}, for $A\in SL_5(\mathbb{Z})\cap u_L\omega_0Tu_R$, we know that
  $$u_L=\begin{pmatrix}
    1 & \dfrac{\langle e^*_{1,3,4,5},Ae_{1,2,3,4}\rangle}{t_1t_2t_3t_4} & \dfrac{\langle e^*_{1,4,5},Ae_{1,2,3}\rangle}{t_1t_2t_3} & \dfrac{\langle e^*_{1,5},Ae_{1,2}\rangle}{t_1t_2} & \dfrac{\langle e^*_{1},Ae_{1}\rangle}{t_1}\\
    & 1 & \dfrac{\langle e^*_{2,4,5},Ae_{1,2,3}\rangle}{t_1t_2t_3} & \dfrac{\langle e^*_{2,5},Ae_{1,2}\rangle}{t_1t_2} & \dfrac{\langle e^*_{2},Ae_{1}\rangle}{t_1}\\
    & & 1 & \dfrac{\langle e^*_{3,5},Ae_{1,2}\rangle}{t_1t_2} & \dfrac{\langle e^*_{3},Ae_{1}\rangle}{t_1}\\
    & & & 1 & \dfrac{\langle e^*_{4},Ae_{1}\rangle}{t_1}\\
    & & & & 1
  \end{pmatrix}$$
  and
  $$u_R=\begin{pmatrix}
    1 & \dfrac{\langle e^*_{5},Ae_{2}\rangle}{t_1} & \dfrac{\langle e^*_{5},Ae_{3}\rangle}{t_1} & \dfrac{\langle e^*_{5},Ae_{4}\rangle}{t_1} & \dfrac{\langle e^*_{5},Ae_{5}\rangle}{t_1}\\
    & 1 & \dfrac{\langle e^*_{4,5},Ae_{1,3}\rangle}{t_1t_2} & \dfrac{\langle e^*_{4,5},Ae_{1,4}\rangle}{t_1t_2} & \dfrac{\langle e^*_{4,5},Ae_{1,5}\rangle}{t_1t_2}\\
    & & 1 & \dfrac{\langle e^*_{3,4,5},Ae_{1,2,4}\rangle}{t_1t_2t_3} & \dfrac{\langle e^*_{3,4,5},Ae_{1,2,5}\rangle}{t_1t_2t_3}\\
    & & & 1 & \dfrac{\langle e^*_{2,3,4,5},Ae_{1,2,3,5}\rangle}{t_1t_2t_3t_4}\\
    & & & & 1
  \end{pmatrix}.$$

  Here, we define 
  $$u_1=d_1x_2+d_2x_3y_1,\ u_2=d_2x_4+d_4x_5y_2,\ u_3=d_4x_7+d_7x_8y_4,\ u_4=d_6x_9y_5+d_9x_{10}y_6,$$
  $$v_1=d_1x_2y_3+d_2y_1,\ v_2=d_2x_4y_5+d_4y_2,\ v_3=d_4x_7y_8+d_7y_4,\ v_4=d_6x_9y_{10}+d_9x_5y_6.$$
  Computing the minors and the entries of $A$ using \eqref{5.65}, for $d_i\ (1\leq i\leq9),\ f$ as in \eqref{5.15} and \eqref{D5}, the Bruhat decomposition of $A\in SL_5(\mathbb{Z})\cap u_L\omega_0Tu_R$ can be written as 
  $$T=\begin{pmatrix}
  d_7d_8d_9f&&&&\\
  &\dfrac{d_4d_5d_6}{d_7}&&&\\
  &&\dfrac{d_2d_3}{d_4d_8}&&\\
  &&&\dfrac{d_1}{d_2d_5d_9}&\\
  &&&&\dfrac{1}{d_1d_3d_6f}
\end{pmatrix},$$

\begin{equation}\label{u_L5}
  u_L=\begin{pmatrix}
    1 & \dfrac{x_1}{d_1} & \dfrac{x_1u_1-d_2x_3}{d_1d_2d_3} & \dfrac{\langle e^*_{1,5},Ae_{1,2}\rangle}{d_4d_5d_6d_8d_9f} & \dfrac{\langle e^*_{1},Ae_{1}\rangle}{d_7d_8d_9f}\\
    & 1 & \dfrac{u_1}{d_2d_3} & \dfrac{\langle e^*_{2,5},Ae_{1,2}\rangle}{d_4d_5d_6d_8d_9f} & \dfrac{\langle e^*_{2},Ae_{1}\rangle}{d_7d_8d_9f}\\
    & & 1 & \dfrac{d_3u_2+d_4d_5x_6y_3}{d_4d_5x_6} & \dfrac{\langle e^*_{3},Ae_{1}\rangle}{d_7d_8d_9f}\\
    & & & 1 & \dfrac{d_5d_6u_3+d_7d_8u_4}{d_7d_8d_9f}\\
    & & & & 1
  \end{pmatrix}
\end{equation} and

\begin{equation}\label{u_R5}
  u_R=\begin{pmatrix}
    1 & \dfrac{y_7}{d_7} & \dfrac{y_8}{d_7d_8} & \dfrac{y_9}{d_7d_8d_9} & \dfrac{y_{10}}{d_7d_8d_9f}\\
    & 1 & \dfrac{v_3}{d_4d_8} & \dfrac{d_5y_9u_3+d_7d_8y_5}{d_4d_5d_8d_9} & \dfrac{\langle e^*_{4,5},Ae_{1,5}\rangle}{d_4d_5d_6d_8d_9f}\\
    & & 1 & \dfrac{d_8v_2+d_2d_5x_8y_9}{d_2d_5d_9} & \dfrac{\langle e^*_{3,4,5},Ae_{1,2,5}\rangle}{d_2d_3d_5d_6d_9f}\\
    & & & 1 & \dfrac{d_5d_9v_1+d_1d_3v_4}{d_1d_3d_6f}\\
    & & & & 1
  \end{pmatrix}.
\end{equation}

Now we consider the fine Kloosterman sum $$\mathcal{S}_{\omega_0}(\mathbf{m},\mathbf{n};D_1,D_2,D_3,D_4f)=\mathcal{S}_{\omega_0}(\mathbf{m},\mathbf{n};d_7d_8d_9,d_4d_5d_6d_8d_9,d_2d_3d_5d_6d_9,d_1d_3d_6,f)$$ with the notation in \eqref{D5}. From \eqref{u_L5} and \eqref{u_R5}, the superdiagonal elements in the unipotent factors of the Bruhat decomposition $\psi_{\mathbf{m}}(u_L)\psi_{\mathbf{n}}(u_R)$ in \eqref{5.45} becomes
\begin{multline*}
  e\left( \dfrac{m_1x_1}{d_1}+\dfrac{m_2u_1}{d_2d_3}+\dfrac{m_3(d_3u_2+d_4d_5x_6y_3)}{d_4d_5d_6}+\dfrac{m_4(d_5d_6u_3+d_7d_8u_4)}{d_7d_8d_9f} \right.\\
  \left.\dfrac{n_1y_7}{d_7}+\dfrac{n_2v_3}{d_4d_8}+\dfrac{n_3(d_8v_2+d_2d_5x_8y_9)}{d_2d_5d_9}+\dfrac{n_4(d_5d_9v_1+d_1d_3v_4)}{d_1d_3d_6f} \right).
\end{multline*}

\end{section}

%%%%%%%%%%%%%%%%%%%%%%%%%%%%%%%%%%%%%%%%%%%%%%%%%%%%%%%%%%%%%%
\section*{Acknowledgment}
We are also sincerely grateful to Professor Eren Mehmet Kiral for his helpful comment and careful reading.

%%%%%%%%%%%%%%%%%%%%%%%%%%%%%%%%%%%%%%%%%%%%%%%%%%%%%%%%%%%%%%

%%%%%%%%%%%%%%%%%%%%%%%%%%%%%%%%%%%%%%%%%%%%%%%%%%%%%%%%%%%%%%

\bigskip
\noindent
\textsc{Suzuho Osonoe}\\
Graduate School of Science and Technology\\
Sophia University, \\
 7-1 Kioi-cho, Chiyoda-ku, Tokyo, 102-8554, Japan \\
 \texttt{s-osonoe-7m6@eagle.sophia.ac.jp}\\

\medskip

\noindent
\textsc{Maki Nakasuji}\\
Department of Information and Communication Science, Faculty of Science, \\
 Sophia University, \\
 7-1 Kioi-cho, Chiyoda-ku, Tokyo, 102-8554, Japan \\
 \texttt{nakasuji@sophia.ac.jp}\\
and\\
Mathematical Institute, \\
Tohoku University, \\
6-3 Aoba, Aramaki, Aoba-ku, Sendai, 980-8578, Japan \\

\end{document}